\documentclass{article}

\usepackage{hyperref}       
\usepackage{url}            
\usepackage{booktabs}       
\usepackage{amsfonts}       
\usepackage{nicefrac}       
\usepackage{microtype}      
\usepackage{pifont}
\usepackage{amssymb}

\usepackage{algorithm}
\usepackage{algorithmic}
\usepackage{hyperref}
\usepackage[square]{natbib}

\usepackage[a4paper,top=3cm,bottom=2cm,left=3cm,right=3cm,marginparwidth=1.75cm]{geometry}

\usepackage{amsmath,amssymb,amsthm}
\usepackage{pgfplots}
\usepackage{color}
\usepackage{flushend}
\usepackage{multirow}
\usepackage{rotating}
\usepackage{appendix}
\usepackage{tikz,pgfplots}
\usepackage{subcaption}

\input{mysymbol.sty}

\newtheorem{assumption}{\hspace{0pt}\bf Assumption}

\newtheorem{definition}{\hspace{0pt}\bf Definition}
\newtheorem{proposition}{\hspace{0pt}\bf Proposition}
\newtheorem{lemma}{\hspace{0pt}\bf Lemma}
\newtheorem{theorem}{\hspace{0pt}\bf Theorem}
\newtheorem{corollary}{\hspace{0pt}\bf Corollary}

\begin{document}

\title{Convergence Rate of $\mathcal{O}(1/k)$ for Optimistic Gradient and Extra-gradient Methods in Smooth Convex-Concave Saddle Point Problems}

\date{}

\author{Aryan Mokhtari\thanks{The authors are in alphabetical order.} \thanks{Department of Electrical and Computer Engineering, University of Texas at Austin, Austin, TX. mokhtari@austin.utexas.edu}, Asuman Ozdaglar$^*$\thanks{Department of Electrical Engineering and Computer Science, Massachusetts Institute of Technology, Cambridge, MA, USA. asuman@mit.edu.}, Sarath Pattathil$^*$\thanks{Department of Electrical Engineering and Computer Science, Massachusetts Institute of Technology, Cambridge, MA, USA. sarathp@mit.edu.}
}

\maketitle
\begin{abstract}
We study the iteration complexity of the optimistic gradient descent-ascent (OGDA) method and the extra-gradient (EG) method for finding a saddle point of a convex-concave unconstrained min-max problem. To do so, we first show that both OGDA and EG can be interpreted as approximate variants of the proximal point method. This is similar to the approach taken in  \citep{nemirovski2004prox} which analyzes EG as an approximation of the `conceptual mirror prox'. In this paper, we highlight how gradients used in OGDA and EG try to approximate the gradient of the Proximal Point method. We then exploit this interpretation to show that both algorithms produce iterates that remain within a bounded set. We further show that the primal dual gap of the averaged iterates generated by both of these algorithms converge with a rate of $\mathcal{O}(1/k)$. Our theoretical analysis is of interest as it provides a the first convergence rate estimate for OGDA in the general convex-concave setting. Moreover, it provides a simple convergence analysis for the EG algorithm in terms of function value without using compactness assumption.  
\end{abstract}


\section{Introduction}\label{sec:intro}

Given a function $f:\reals^m\times \reals^n \to \reals$, we consider finding a saddle point of the problem
\begin{equation}\label{main_prob}
\min_{\bbx \in \reals^m}\max_{\bby \in \reals^n} \ f(\bbx,\bby),
\end{equation}
where a saddle point of Problem \eqref{main_prob} is defined as a pair $(\bbx^*, \bby^*) \in \reals^m \times \reals^n $ that satisfies
$$ f(\bbx^*, \bby) \leq f(\bbx^*, \bby^*) \leq f(\bbx, \bby^*) $$
for all $\bbx \in \reals^m, \bby \in \reals^n$. Throughout the paper, we assume that the function $f(\bbx,\bby)$ is \textit{convex-concave}, i.e., for any $\bby \in \mathbb{R}^n$, the function $f(\bbx,\bby)$ is a convex function of $\bbx$ and for any $\bbx \in \mathbb{R}^m$, the function $f(\bbx,\bby)$ is a concave function of $\bby$. This formulation arises in several areas, including zero-sum games \citep{basar1999dynamic}, robust optimization \citep{ben2009robust}, robust control \citep{hast2013pid}  and more recently in machine learning in the context of Generative Adversarial Networks (GANs) (see \citep{goodfellow2014generative} for an introduction to GANs and \citep{pmlr-v70-arjovsky17a} for the formulation of Wasserstein GANs).

Our goal in this paper is to analyze the convergence rate of some discrete-time gradient based optimization algorithms for finding a saddle point of Problem \eqref{main_prob} in the convex-concave case. In particular, we focus on Extra-gradient (EG) and Optimistic Gradient Descent Ascent (OGDA) methods because of their widespread use for training GANs (see \citep{DBLP:journals/corr/abs-1711-00141, DBLP:journals/corr/abs-1802-06132}). EG method is a classical algorithm for solving saddle point problems introduced by \citet{korpelevich1976extragradient}. Its linear rate of convergence for smooth and strongly convex-strongly concave functions $f(\bbx, \bby)$ \footnote{Note that when we state that $f(\bbx,\bby)$ is strongly convex-strongly concave, it means that $f(\cdot, \bby)$ is strongly convex for all $\bby \in \reals^n$ and $f(\bbx, \cdot)$ is strongly concave for all $\bbx \in \reals^m$.} and bilinear functions, i.e., $f(\bbx, \bby) = \bbx^{\top}\bbA \bby$ (where $\bbA$ is a square, full rank matrix), was established in \citet{korpelevich1976extragradient} as well as the variational inequality literature (see \citep{tseng_1} and \citep{facchinei2007finite}). Its $\mathcal{O}(1/k)$ convergence rate for the constrained convex-concave setting was first established by \cite{nemirovski2004prox} under the assumption that the feasible set is convex and compact.\footnote{The result in \citep{nemirovski2004prox} shows a $\mathcal{O}(1/k)$ convergence rate for the mirror-prox algorithm which specializes to the EG method for the Euclidean case.} \cite{monteiro2010complexity} established a similar $\mathcal{O}(1/k)$ convergence rate for EG without assuming compactness of the feasible set by using a new termination criterion that relies on enlargement of the operator of the VI reformulation of the saddle point problem defined in \citep{burachik1997enlargement}. OGDA was introduced by \cite{popov1980modification}, as a variant of the Extragradient method, and has gained popularity recently due to its performance in training GANs (see \citep{DBLP:journals/corr/abs-1711-00141}). To the best of our knowledge, iteration complexity of OGDA for the convex-concave case has not been studied before. 

In this paper, we provide a unified convergence analysis for establishing a sublinear convergence rate of $\mathcal{O}(1/k)$ in terms of the function value difference of the averaged iterates and a saddle point for both OGDA and EG for convex-concave saddle point problems. Our analysis holds for unconstrained problems and does not require boundedness of the feasible set, and it establishes rate results using the function value differences as used in \citep{nemirovski2004prox} (suitably redefined for an unconstrained feasible set, see Section \ref{sec:EG}). Therefore, we get convergence of the EG method in unconstrained spaces without using the modified termination (error) criterion proposed in \citep{monteiro2010complexity}. The key idea of our approach is to view both OGDA and EG iterates as approximations of the iterates of the proximal point method that was first introduced by \cite{martinet1970breve} and later studied by \cite{rockafellar1976monotone}. 
We would like to add that the idea of interpreting EG as an approximation of the Proximal Point method was first studied in \citep{nemirovski2004prox}. He considers the conceptual mirror prox, which is similar to the proximal point method, and shows that the mirror prox algorithm (of which EG is a special case) provides a good implementable approximation to this method. Further, \cite{monteiro2010complexity} use a similar interpretation and propose the Hybrid Proximal Extragradient method to establish the convergence of EG in unbounded settings using a different convergence criteria. More recently, \cite{mokhtari2019unified} study both OGDA and EG as approximations of proximal point method and analyze these algorithms for bilinear and strongly convex-strongly concave problems.

More specifically, we first consider a proximal point method with error and establish some key properties of its iterates. We then focus on OGDA as an approximation of proximal point method and use this connection to show that the iterates of OGDA remain in a compact set. We incorporate this result to prove a sublinear convergence rate of $\mathcal{O}(1/k)$ for the primal-dual gap of the averaged iterates generated by the OGDA update. We next consider EG where two gradient pairs are used in each iteration, one to compute a midpoint and other to find the new iterate using the gradient of the midpoint. Our first step again is to show boundedness of the iterates generated by EG. We then approximate the evolution of the midpoints using a proximal point method and use this approximation to establish $\mathcal{O}(1/k)$ convergence rate for the function value of the averaged iterates generated by EG. 
As the convergence results of EG have already been established in papers including \citep{nemirovski2004prox} and \cite{monteiro2010complexity}, we relegate the proofs of Lemmas and Theorems corresponding to EG to the Appendix.

\subsection*{Related Work}

Several recent papers have studied the convergence rate of OGDA and EG for the case when the objective function is bilinear 
or strongly convex-strongly concave. 
\cite{DBLP:journals/corr/abs-1711-00141} showed the convergence of the OGDA iterates to a neighborhood of the solution when the objective function is bilinear.  \cite{DBLP:journals/corr/abs-1802-06132} used a dynamical system approach to prove the linear convergence of the OGDA method for the special case when $f(\bbx,\bby) = \bbx^{\top} \bbA \bby$ and the matrix $\bbA$ is square and full rank. They also presented a linear convergence rate of the vanilla Gradient Ascent Descent (GDA) method when the objective function $f(\bbx, \bby)$ is strongly convex-strongly concave. 
\cite{gidel2018variational} considered a variant of the EG method, relating it to OGDA updates, and showed the linear convergence of the corresponding EG iterates in the case where $f(\bbx, \bby)$ is strongly convex-strongly concave (though without showing the convergence rate for the OGDA iterates).
Optimistic gradient methods have also been studied in the context of convex online learning \citep{DBLP:journals/jmlr/ChiangYLMLJZ12,DBLP:conf/colt/RakhlinS13,rakhlin2013optimization}.


\cite{nedic2009subgradient} analyzed the (sub)Gradient Descent Ascent (GDA) algorithm for convex-concave saddle point problems when the (sub)gradients are bounded over the constraint set, showing a convergence rate of $\mathcal{O}(1/ \sqrt{k})$  in terms of the function value difference of the averaged iterates and a saddle point.


\cite{chambolle2011first} focused on a particular case of the saddle point problem where the coupling term in the objective function is bilinear, i.e., $f(\bbx, \bby) = G(\bbx) + \bbx^{\top} \bbK \bby - H(\bby)$ with $G$ and $H$ convex functions. They proposed a proximal point based algorithm which converges at a rate $\mathcal{O}(1/k)$ and further showed linear convergence when the functions $G$ and $H$ are strongly convex. 
\cite{chen2014optimal} proposed an accelerated variant of this algorithm when $G$ is smooth and showed an optimal rate of $\mathcal(\frac{L_G}{k^2} + \frac{L_K}{k} )$, where $L_G$ and $L_K$ are the smoothness parameters of $G$ and the norm of the linear operator $K$ respectively. 
When the functions $G$ and $H$ are strongly convex, primal-dual gradient-type methods converge linearly, as shown in  \cite{chen1997convergence,bauschke2011convex}.
Further, \cite{DBLP:journals/corr/abs-1802-01504} showed that GDA achieves a linear convergence rate in this linearly coupled setting when $G$ is convex and $H$ is strongly convex. 

For the case that $f(\bbx, \bby)$ is strongly concave with respect to $\bby$, but possibly nonconvex with respect to $\bbx$, \cite{sanjabi2018convergence} provided convergence to a first-order stationary point using an algorithm that requires running multiple updates with respect to $\bby$ at each step. 

\medskip\noindent{\bf Notation.\quad} Lowercase boldface $\bbv$ denotes a vector and uppercase boldface $\bbA$ denotes a matrix. We use $\|\bbv\|$ to denote the Euclidean norm of vector $\bbv$. Given a multi-input function $f(\bbx,\bby)$, its gradient with respect to $\bbx$ and $\bby$ at points $(\bbx_0,\bby_0)$ are denoted by $\nabla_\bbx f(\bbx_0,\bby_0)$ and $\nabla_\bby f(\bbx_0,\bby_0)$, respectively. We refer to the largest and smallest eigenvalues of a matrix $\bbA$ by $\lambda_{\max}(\bbA)$ and  $\lambda_{\min}(\bbA)$, respectively.


\section{Preliminaries}\label{sec:preliminaries}

In this section we present properties and notations used in our results.

\begin{definition} \label{def:lips_grad}
A function $\phi: \reals^n\to\reals$ is $L$-smooth if it has $L$-Lipschitz continuous gradients on $\reals^n$, i.e., for any $\bbx, \hbx \in \reals^n$, we have
\begin{align}
|| \nabla \phi(\bbx) - \nabla \phi(\hbx) || \leq L ||\bbx - \hbx||. \nonumber
\end{align}
\end{definition}

%

\begin{definition}  \label{def:strong_convex}
A continuously differentiable function $\phi: \reals^n\to\reals$ is convex on $\mathbb{R}^n$ if for any $\bbx, \hbx \in \reals^n$, we have
\begin{align}
\phi(\hbx) \geq \phi(\bbx) + \nabla \phi(\bbx) ^T (\hbx - \bbx).  \nonumber
\end{align}
Further, $\phi(\bbx)$ is concave if $-\phi(\bbx)$ is convex.
\end{definition} 

\begin{definition}  \label{def:saddle_point}
The pair $(\bbx^*, \bby^*)$ is a saddle point of a convex-concave function $f(\bbx, \bby)$, if for any $\bbx \in \reals^n$ and $ \bby \in \reals^m$, we have
\begin{align}
f(\bbx^*, \bby) \leq f(\bbx^*, \bby^*) \leq f(\bbx, \bby^*). \nonumber
\end{align}
\end{definition} 


Throughout the paper, we will assume that the following conditions are satisfied.

\begin{assumption}
\label{ass:convex_concave}
The function $f(\bbx, \bby)$ is continuously differentiable in $\bbx$ and $\bby$. Further, for any $\bby \in \mathbb{R}^n$, the function $f(\bbx,\bby)$ is a convex function of $\bbx$ and for any $\bbx \in \mathbb{R}^m$, the function $f(\bbx,\bby)$ is a concave function of $\bby$. 
\end{assumption}

\begin{assumption}
\label{ass:scsc}
The gradient $\nabla_{\bbx}f(\bbx, \bby)$, is $L_{xx}$-Lipschitz with respect to $\bbx$ and $L_{xy}$-Lipschitz with respect to $\bby$ and the gradient $\nabla_{\bby}f(\bbx, \bby)$, is $L_{yy}$-Lipschitz with respect to $\bby$ and $L_{yx}$-Lipschitz with respect to $\bbx$, i.e., \begin{align*}
\| \nabla_{\bbx}f(\bbx_1, \bby) - \nabla_{\bbx}f(\bbx_2, \bby)\| &\leq L_{xx}\| \bbx_1 - \bbx_2 \| \quad \forall \: \bby,  \\
\| \nabla_{\bbx}f(\bbx, \bby_1) - \nabla_{\bbx}f(\bbx, \bby_2)\| &\leq L_{xy}\| \bby_1 - \bby_2 \| \quad \forall \: \bbx,  \\
\| \nabla_{\bby}f(\bbx, \bby_1) - \nabla_{\bby}f(\bbx, \bby_2)| &\leq L_{yy}\| \bby_1 - \bby_2 \| \quad {\forall} \: \bbx, \\
\| \nabla_{\bby}f(\bbx_1, \bby) - \nabla_{\bby}f(\bbx_2, \bby)| &\leq L_{yx}\| \bbx_1 - \bbx_2 \| \quad {\forall} \: \bby.
\end{align*}
We define $L := 2 \times \max \{ L_{xx},  L_{xy}, L_{yx}, L_{yy} \}$. \footnote{  In this definition we need an additional factor of 2 because in the analysis we use $L$ as the Lipschitz continuity of the operator $F(\cdot) = [\nabla_\bbx f(\cdot); -\nabla_\bby f(\cdot)]$.}
\end{assumption}

\begin{assumption}
\label{ass:sol_set}
The solution set $\mathcal{Z}^*$ defined as
\begin{align}
\label{eq:defi_soln_set}
\mathcal{Z}^* := \{ [\bbx; \bby] \in \reals^{n+m} : (\bbx, \bby) \text{ is a saddle point of Problem } \eqref{main_prob} \},
\end{align}
is nonempty.
\end{assumption}
In the following sections, we present and analyze three different iterative algorithms for solving the saddle point problem introduced in~\eqref{main_prob}. The $k^{th}$ iterates of these algorithms are denoted by $(\bbx_k, \bby_k)$. 
We denote the averaged (ergodic) iterates by $\hat{\bbx}_k, \hat{\bby}_k$, defined as follows:
\begin{align}\label{avg_iterate}
\hat{\bbx}_k = \frac{1}{k} \sum_{i = 1}^{k} \bbx_i , \qquad 
\hat{\bby}_k = \frac{1}{k} \sum_{i = 1}^{k} \bby_i .
\end{align}

In our convergence analysis, we use a variational inequality approach in which we define the vector $\bbz = [\bbx; \bby]\in \reals^{n+m}$ as our decision variable and define the operator  $F: \reals^{m+n}\to \reals^{m+n}$ as
\begin{align}\label{eq:operator_def}
F(\bbz) = [\nabla_{\bbx} f(\bbx, \bby) ;  - \nabla_{\bby} f(\bbx, \bby)  ].
\end{align}

In the following lemma we characterize the properties of operator $F$ in \eqref{eq:operator_def} when the conditions in Assumptions \ref{ass:convex_concave} and \ref{ass:scsc} are satisfied. We would like to emphasize that the following lemma is well-known -- see, e.g., \cite{nemirovski2004prox} -- and we state it for completeness.
\begin{lemma}
\label{lemma:operator_monotone}
Let $F(\cdot)$ be defined as in Equation \eqref{eq:operator_def}. Suppose Assumptions \ref{ass:convex_concave} and \ref{ass:scsc} hold. Then \\
(a) $F$ is a monotone operator,  i.e., for any $\bbz_1,\bbz_2\in\reals^{m+n}$, we have
\begin{align*}
\langle F(\bbz_1) - F(\bbz_2) , \bbz_1 - \bbz_2 \rangle \geq 0.
\end{align*}
(b) $F$ is an $L$-Lipschitz continuous operator, i.e., for any $\bbz_1,\bbz_2\in\reals^{m+n}$, we have
\begin{align*}
\| F(\bbz_1) - F(\bbz_2)\|  \leq L \| \bbz_1 - \bbz_2 \|.
\end{align*}
(c) For all $\bbz^* \in \mathcal{Z}^*$, we have $F(\bbz^*) = 0$.
\end{lemma}

According to Lemma~\ref{lemma:operator_monotone}, when $f$ is convex-concave and smooth, the operator $F$ defined in \eqref{eq:operator_def} is monotone and Lipschitz. The third result in Lemma~\ref{lemma:operator_monotone} shows that any saddle point of problem~\eqref{main_prob} satisfies the first-order optimality condition, i.e $\forall \ (\bbx^*, \bby^*) \in \mathcal{Z}^*$, we have:
\begin{align}
\nabla_{\bbx}f(\bbx^*, \bby^*) = 0 \qquad \nabla_{\bby}f(\bbx^*, \bby^*)  
\end{align}

Before presenting our main results, we state the following well known result (see for example \cite{nemirovski2004prox}) which will be used later in the analysis of OGDA and EG. We present the proof here for completeness.

\begin{proposition}
\label{lemma:iter_erg_ineq}
Recall the definition of the operator $F(\cdot)$ in \eqref{eq:operator_def} and the points $\hbx_k,\hby_k$ in \eqref{avg_iterate}.  Suppose Assumptions \ref{ass:convex_concave} and \ref{ass:sol_set}  hold. Then for any $\bbz = [\bbx; \bby] \in \reals^{m+n}$, we have
\begin{align}
f(\hat{\bbx}_N, \bby) - f(\bbx, \hat{\bby}_N) \leq \frac{1}{N} \sum_{k = 1}^{N} F(\bbz_{k}) ^{\top}  (\bbz_{k} - \bbz)
\end{align}
\end{proposition}

\begin{proof}
Using the definition of the operator $F$, we can write
\begin{align}
\frac{1}{N} \sum_{k = 1}^{N} F(\bbz_{k}) ^{\top}  (\bbz_{k} - \bbz) &= \frac{1}{N} \sum_{k = 1}^{N} [ \nabla_{\bbx}f(\bbx_k, \bby_k)^{\top} (\bbx_{k} - \bbx ) + \nabla_{\bby}f(\bbx_k, \bby_k)^{\top} (\bby - \bby_k ) ] \nonumber \\
&\geq \frac{1}{N} \sum_{k = 1}^{N} [f(\bbx_k, \bby_k) - f(\bbx, \bby_k) + f(\bbx_k, \bby) - f(\bbx_k, \bby_k) ] \nonumber \\
&= \frac{1}{N} \sum_{k = 1}^{N} [ f(\bbx_k, \bby)  - f(\bbx, \bby_k) ], \label{eq:bw_lemma_1}
\end{align}
where the inequality holds due to the fact that $f$ is convex-concave. Using convexity of $f$ with respect to $\bbx$ and concavity of $f$ with respect to $\bby$, we have
\begin{align}
\frac{1}{N} \sum_{k = 1}^{N} f(\bbx_k, \bby) \geq   f(\hat{\bbx}_N, \bby), \qquad 
 \frac{1}{N} \sum_{k = 1}^{N} f(\bbx, \bby_k ) \leq f(\bbx, \hat{\bby}_N).
 \label{eq:bw_lemma_2}
\end{align}
Combining inequalities \eqref{eq:bw_lemma_1} and \eqref{eq:bw_lemma_2} yields
\begin{align}
\frac{1}{N} \sum_{k = 1}^{N} F(\bbz_{k}) ^{\top}  (\bbz_{k} - \bbz) &\geq f(\hat{\bbx}_N, \bby) - f(\bbx, \hat{\bby}_N), \nonumber
\end{align}
completing the proof.
\end{proof}


\section{Proximal point method with error} One of the classical algorithms studied for solving the saddle point problem in \eqref{main_prob} is the Proximal Point (PP) method, introduced in \cite{martinet1970breve} and studied in \cite{rockafellar1976monotone}. The PP method  generates the iterate $\{ \bbx_{k+1}, \bby_{k+1} \}$ which is defined as the unique solution to the saddle point problem\footnote{Again $ \{\bbx_{k+1},\bby_{k+1}\} $ is  unique since the objective function of problem \eqref{eq:prox_point_min_max_update_0} is strongly convex in $\bbx$ and strongly concave in $\bby$}
\begin{align}\label{eq:prox_point_min_max_update_0}
\min_{\bbx \in \reals^m} \max_{\bby \in \reals^n} \left\{ f(\bbx,\bby) + \frac{1}{2\eta}\|\bbx-\bbx_k\|^2- \frac{1}{2\eta}\|\bby-\bby_k\|^2\right\}.
\end{align}
It can be verified that if the pair $\{\bbx_{k+1},\bby_{k+1}\} $ is the solution of problem \eqref{eq:prox_point_min_max_update_0}, then $\bbx_{k+1}$ and $\bby_{k+1} $ satisfy 
\begin{align}
\bbx_{k+1} &= \argmin_{\bbx \in \reals^m}\left\{ f(\bbx,\bby_{k+1}) + \frac{1}{2\eta}\|\bbx-\bbx_k\|^2\right\},\label{eq:prox_point_min_max_update_1}\\
\bby_{k+1} &= \argmax_{\bby \in \reals^n}\left\{ f(\bbx_{k+1},\bby) - \frac{1}{2\eta}\|\bby-\bby_k\|^2\right\}. \label{eq:prox_point_min_max_update_2}
\end{align}
Using the optimality conditions of the updates in \eqref{eq:prox_point_min_max_update_1} and \eqref{eq:prox_point_min_max_update_2} (which are necessary and sufficient since the problems in \eqref{eq:prox_point_min_max_update_1} and \eqref{eq:prox_point_min_max_update_2} are strongly convex and strongly concave, respectively), the update of the PP method for the saddle point problem in \eqref{main_prob} can be written as 
%

\begin{align}\label{eq:proximal_point_update}
\bbx_{k+1} &= \bbx_{k} - \eta \nabla_{\bbx} f (\bbx_{k+1},  \bby_{k+1}), \nonumber \\
\bby_{k+1} &= \bby_k + \eta \nabla_{\bby} f (\bbx_{k+1},  \bby_{k+1}). 
\end{align}
It is well-known that the proximal point method achieves a sublinear rate of $\mathcal{O}(1/k)$ when $k$ is the number of iterations for convex minimization and for solving monotone variational inequalities (see \cite{guler1991convergence,guler1992new, bruck1977weak, teboulle1997convergence, nemirovski2004prox}). Note that \cite{nemirovski2004prox} in fact analyzed the conceptual mirror prox (the proximal point method) as a building block to analyze the mirror-prox algorithm. For completeness, we present the convergence rate of the proximal point method for convex-concave saddle point problems in the following theorem (see Appendix \ref{appx:proof_of_thm_pp} for the proof).

\begin{theorem}\label{theorem:prox_point_rate}
Suppose Assumption \ref{ass:convex_concave} holds. Let $\{ \bbx_k, \bby_k \}$ be the iterates generated by the updates in \eqref{eq:proximal_point_update}. Consider the definition of the averaged iterates $\hbx_k,\hby_k$ in \eqref{avg_iterate}. Then for all $k \geq 1$, we have
\begin{align}
 \left| f(\hat{\bbx}_k,  \hat{\bby}_k) -  f(\bbx^*,   \bby^*) \right| \leq \frac{ \| \bbx_0 - \bbx^* \|^2 + \| \bby_0 - \bby^* \|^2 }{\eta k}.
\end{align}
\end{theorem}

The result in Theorem~\ref{theorem:prox_point_rate} shows that by following the update of proximal point method the gap between the function value for the averaged iterates $(\hat{\bbx}_k,  \hat{\bby}_k)$ and the function value for a saddle point $(\bbx^*,   \bby^*)$ of the problem \eqref{main_prob} approaches zero at a sublinear rate of $\mathcal{O}(1/k)$.

Our goal is to provide similar convergence rate estimates for OGDA and EG using the fact that these two methods can be interpreted as approximate versions of the proximal point method. To do so, let us first rewrite the update of the proximal point method given in  \eqref{eq:proximal_point_update} as 
\begin{align}
\bbz_{k+1} = \bbz_k - \eta F(\bbz_{k+1}),
\end{align}
where $\bbz=[\bbx; \bby] \in \mathbb{R}^{m+n}$ and the operator $F$ is defined in \eqref{eq:operator_def}. In the following proposition, we establish a relation for the iterates of a proximal point method with error. This relation will be used later for our analysis of OGDA and EG methods.
%
%

\begin{proposition}\label{lemma_general_pp_approx}
Consider the sequence of iterates $\{\bbz_k\}\in \reals^{n+m}$ generated by the following update
\begin{align}\label{VI_pp_update}
\bbz_{k+1} = \bbz_k - \eta F(\bbz_{k+1}) + \bbvarepsilon_k,
\end{align}
where $F:\reals^{n+m}\to\reals^{n+m}$ is a monotone and Lipschitz continuous operator, $\bbvarepsilon_k\in \reals^{n+m}$ is an arbitrary vector, and $\eta$ is a positive constant. Then for any $\bbz \in \reals^{n+m}$ and for each $k\geq1$ we have
\begin{align}\label{proof_mid_lemma_600}
&F(\bbz_{k+1}) ^{\top} (\bbz_{k+1} - \bbz) \nonumber \\
& \quad
= \frac{1}{2\eta}  \|\bbz_k - \bbz\|^2 - \frac{1}{2\eta}  \|\bbz_{k+1} - \bbz\|^2 -\frac{1}{2\eta}  \|\bbz_{k+1} - \bbz_k\|^2 + \frac{1}{\eta} {\bbvarepsilon_k}^{\top} (\bbz_{k+1} - \bbz).
\end{align}
\end{proposition}

\begin{proof}
According to the update in \eqref{VI_pp_update}, we can show that for any $\bbz\in \reals^{m+n}$ we have 
\begin{align}\label{proof_mid_lemma_100}
\|\bbz_{k+1} - \bbz\|^2 = \|\bbz_k - \bbz\|^2 -  2\eta (\bbz_k - \bbz)^\top  F(\bbz_{k+1})  & + \eta^2 \|F(\bbz_{k+1})  \|^2 + \|\bbvarepsilon_k\|^2 \nonumber \\
&+ 2 {\bbvarepsilon_k}^\top (\bbz_k - \bbz - \eta F(\bbz_{k+1}) ).
\end{align}
We add and subtract the inner product $2\eta\bbz_{k+1}^\top  F(\bbz_{k+1})$ to the right hand side and regroup the terms to obtain
\begin{align}\label{proof_mid_lemma_200}
\|\bbz_{k+1} - \bbz\|^2 
&= \|\bbz_k - \bbz\|^2 - 2\eta (\bbz_{k+1} - \bbz)^{\top} F(\bbz_{k+1})- 2\eta (\bbx_{k} - \bbx_{k+1})^{\top}  F(\bbz_{k+1})\nonumber \\
&\quad + \eta^2 \| F(\bbz_{k+1}) \|^2 + \|\bbvarepsilon_k\|^2   + 2 {\bbvarepsilon_k}^T (\bbz_k - \bbz - \eta F(\bbz_{k+1}) ) .
\end{align}
Replacing $F(\bbz_{k+1})$ with $(1/\eta)(-\bbz_{k+1}+\bbz_k+\bbvarepsilon_k)$, we obtain
\begin{align}
\label{proof_mid_lemma_500}
&\|\bbz_{k+1} - \bbz\|^2 \nonumber\\
&= \|\bbz_k - \bbz\|^2 - 2\eta (\bbz_{k+1} - \bbz)^{\top} F(\bbz_{k+1})+ 2(\bbz_{k} - \bbz_{k+1})^\top (\bbz_{k+1} - \bbz_k -\bbvarepsilon_k) \nonumber \\
& \qquad + \|\bbz_{k+1} - \bbz_k -\bbvarepsilon_k \|^2  + \|\bbvarepsilon_k\|^2 + 2 {\bbvarepsilon_k}^T (\bbz_{k+1} - \bbz - \bbvarepsilon_k ) \nonumber \\
&=  \|\bbz_k - \bbz\|^2 - 2\eta (\bbz_{k+1} - \bbz)^{\top} F(\bbz_{k+1})   - \|\bbz_{k+1} - \bbz_{k}\|^2 + 2 {\bbvarepsilon_k}^T (\bbz_{k+1} - \bbz).
\end{align}
On rearranging the terms, we obtain the following inequality:
\begin{align}
&F(\bbz_{k+1}) ^{\top} (\bbz_{k+1} - \bbz) 
\nonumber \\
& \quad= \frac{1}{2\eta}  \|\bbz_k - \bbz\|^2 - \frac{1}{2\eta}   \|\bbz_{k+1} - \bbz\|^2 -\frac{1}{2\eta}  \|\bbz_{k+1} - \bbz_k\|^2 + \frac{1}{\eta} {\bbvarepsilon_k}^T (\bbz_{k+1} - \bbz),
\end{align}
and the proof is complete.
\end{proof}

\section{Optimistic Gradient Descent Ascent}
In this section, we focus on analyzing the performance of optimistic gradient descent ascent (OGDA) for finding a saddle point of a general smooth convex-concave function. It has been shown that the OGDA method achieves the same iteration complexity as the proximal point method for both strongly convex-strongly concave and bilinear problems; see \cite{DBLP:journals/corr/abs-1802-06132}, \cite{gidel2018variational}, \cite{mokhtari2019unified}. However, its iteration complexity for a general smooth convex-concave case has not been established to the best of our knowledge. In this section, 
we show that the function value of the averaged iterate generated by the OGDA method converges to the function value at a saddle point at a rate of $\mathcal{O}(1/k)$, 
which matches the convergence rate of the proximal point method shown in Theorem \ref{theorem:prox_point_rate}.

Given a stepsize $\eta>0$, the OGDA method updates the iterates $\bbx_k$ and $\bby_k$ for each $k \geq 0$ as
\begin{align}\label{OGDA_update}
\bbx_{k+1} &= \bbx_k - 2\eta \nabla_{\bbx} f \left( \bbx_k , \bby_k  \right) + \eta \nabla_{\bbx} f \left( \bbx_{k-1} , \bby_{k-1}  \right),\nonumber\\ 
\bby_{k+1} &= \bby_k + 2\eta \nabla_{\bby} f \left( \bbx_k , \bby_k  \right) - \eta \nabla_{\bby} f \left( \bbx_{k-1} , \bby_{k-1}  \right)
\end{align}
with the initial conditions $\bbx_0 = \bbx_{-1}$ and $\bby_0 = \bby_{-1}$. The main difference between the updates of OGDA in \eqref{OGDA_update} and the gradient descent ascent (GDA) method is in the additional ``momentum" terms $-\eta (\nabla_{\bbx} f \left( \bbx_k , \bby_k  \right)-\nabla_{\bbx} f \left( \bbx_{k-1} , \bby_{k-1}  \right))$ and $\eta (\nabla_{\bby} f \left( \bbx_k , \bby_k  \right)-\nabla_{\bby} f \left( \bbx_{k-1} , \bby_{k-1}  \right))$. This additional term makes the update of OGDA a better approximation to the update of the proximal point method compared to the update of the GDA; for more details we refer readers to Proposition 1 in \cite{mokhtari2019unified}.


To establish the convergence rate of  OGDA for convex-concave problems, we first illustrate the connection between the updates of proximal point method and OGDA. Note that using the definitions of the vector  $\bbz = [\bbx; \bby]\in \reals^{n+m}$ and the operator $F(\bbz) = [\nabla_{\bbx} f(\bbx, \bby) ;  - \nabla_{\bby} f(\bbx, \bby)  ]\in \reals^{n+m}$, we can rewrite the update of the OGDA algorithm at iteration $k$ as 
\begin{align}
\label{eq:OGDA_update_VI}
\bbz_{k+1} = \bbz_k - 2\eta F (\bbz_k) + \eta F (\bbz_{k-1}).
\end{align}
Considering this expression, we can also write the update of OGDA as an approximation of the proximal point update, i.e., 
\begin{align} \label{eq:PP_E_OGDA}
\bbz_{k+1} = \bbz_k - \eta F (\bbz_{k+1}) + \bbvarepsilon_k,
\end{align}
where the error vector $\bbvarepsilon_k$ is given by
\begin{align}\label{eq:error_vec_ogda}
\bbvarepsilon_k = \eta[(F(\bbz_{k+1}) - F(\bbz_k)) - (F(\bbz_{k}) - F(\bbz_{k-1}))] .
\end{align}

To derive the convergence rate of OGDA for the unconstrained problem in \eqref{main_prob}, we first use the result in Proposition~\ref{lemma_general_pp_approx} to derive a result for the specific case of OGDA updates. We then show that the iterates generated by the OGDA method remain in a bounded set. This is done in the following lemma (Note that boundedness of OGDA iterates can be deduced from \citep{popov1980modification}, whereas a result similar to Lemma \ref{lemma:ogda_bounded_set}(b) was shown in a recent independent paper by \cite{malitsky2018forward}).

\begin{lemma}\label{lemma:ogda_bounded_set}
Let $\{ \bbz_k \}$ be the iterates generated by the optimistic gradient descent ascent (OGDA) method introduced in \eqref{eq:OGDA_update_VI} with the initial conditions $\bbx_0 = \bbx_{-1}$ and $\bby_0 = \bby_{-1}$ (i.e. $\bbz_0 = \bbz_{-1}$). If  Assumptions~\ref{ass:convex_concave},~\ref{ass:scsc}, and \ref{ass:sol_set} hold and the stepsize $\eta$ satisfies the condition $0<\eta \leq \frac{1}{2L}$, then: \\
(a) The iterates $\{ \bbz_k \}$ satisfy the following relation:
\begin{align}\label{proof_so_far_3000_1}
&F(\bbz_{k+1})^{\top}(\bbz_{k+1} - \bbz) \nonumber\\
& \leq \frac{1}{2\eta} \| \bbz_{k} - \bbz\|^2 - \frac{1}{2\eta} \| \bbz_{k+1} - \bbz\|^2 - \frac{L}{2} \| \bbz_{k+1} - \bbz_{k} \|^2 +  \frac{L}{2} \| \bbz_{k} - \bbz_{k-1}\|^2 \nonumber \\
& \qquad +(F(\bbz_{k+1}) - F(\bbz_k))^{\top}(\bbz_{k+1} - \bbz) - (F(\bbz_{k}) - F(\bbz_{k-1}))^{\top}(\bbz_{k} - \bbz).
\end{align}

\noindent (b) The iterates $\{ \bbz_k \}$ stay within the compact set $\ccalD$ defined as 
\begin{equation}\label{set_D}
\ccalD:=\{(\bbx,\bby)\mid \| \bbx - \bbx^* \|^2 + \| \bby - \bby^* \|^2\leq 2\left( \| \bbx_0 - \bbx^* \|^2 + \| \bby_0 - \bby^* \|^2 \right) \},
\end{equation}
where $(\bbx^*,\bby^*) = \bbz^* \in \mathcal{Z}^*$ is a saddle point of the problem defined in \eqref{main_prob}.
\end{lemma}

\begin{proof}

Since OGDA iterates satisfy Equation \eqref{eq:PP_E_OGDA} with the error vector $\bbvarepsilon_k$ given in Equation \eqref{eq:error_vec_ogda}, using Proposition~\ref{lemma_general_pp_approx} with this error vector $\bbvarepsilon_k$ leads to
\begin{align}
&F(\bbz_{k+1}) ^{\top} (\bbz_{k+1} - \bbz) 
\nonumber \\
& = \frac{1}{2\eta}  \|\bbz_k - \bbz\|^2 - \frac{1}{2\eta} \|\bbz_{k+1} - \bbz\|^2 -\frac{1}{2\eta}  \|\bbz_{k+1} - \bbz_k\|^2 \nonumber\\
&\quad+ (F(\bbz_{k+1}) - F(\bbz_k)) ^\top (\bbz_{k+1} - \bbz) - (F(\bbz_{k}) - F(\bbz_{k-1}))^\top(\bbz_{k+1} - \bbz).
\end{align}
We add and subtract the inner product $(F(\bbz_{k}) - F(\bbz_{k-1}))^{\top}(\bbz_{k} - \bbz)$ to the right hand side of the preceding relation to obtain
\begin{align}\label{proof_so_far_1000}
&F(\bbz_{k+1}) ^{\top} (\bbz_{k+1} - \bbz) 
\nonumber \\
& = \frac{1}{2\eta}  \|\bbz_k - \bbz\|^2 - \frac{1}{2\eta} \|\bbz_{k+1} - \bbz\|^2 -\frac{1}{2\eta}  \|\bbz_{k+1} - \bbz_k\|^2 \nonumber\\
&\quad+ (F(\bbz_{k+1}) - F(\bbz_k)) ^\top (\bbz_{k+1} - \bbz) - 
(F(\bbz_{k}) - F(\bbz_{k-1}))^{\top}(\bbz_{k} - \bbz) \nonumber\\
&\quad
+(F(\bbz_{k}) - F(\bbz_{k-1}))^{\top}(\bbz_{k} - \bbz_{k+1}).
\end{align}
Note that $(F(\bbz_{k}) - F(\bbz_{k-1}))^{\top}(\bbz_{k} - \bbz_{k+1})$  can be upper bounded by
\begin{align}\label{sth_needed}
(F(\bbz_{k}) - F(\bbz_{k-1}))^{\top}(\bbz_{k} - \bbz_{k+1})
&\leq \|F(\bbz_{k}) - F(\bbz_{k-1})\| \|\bbz_{k} - \bbz_{k+1}\|\nonumber\\
&\leq L  \| \bbz_{k} - \bbz_{k-1}\|\|\bbz_{k} - \bbz_{k+1}\|\nonumber\\
&\leq \frac{L}{2}  \| \bbz_{k} - \bbz_{k-1}\|^2+\frac{L}{2}\|\bbz_{k} - \bbz_{k+1}\|^2,
\end{align}
where the second inequality holds due to Lipschitz continuity of the operator $F$ (Lemma \ref{lemma:operator_monotone}(b)) and the last inequality holds due to Young's inequality.\footnote{We use the following form of Young's inequality throughout the paper:$$\bba^{\top} \bbb \leq \frac{\| \bba \|^2}{2} + \frac{\| \bbb\|^2}{2}$$} Replacing $(F(\bbz_{k}) - F(\bbz_{k-1}))^{\top}(\bbz_{k} - \bbz_{k+1})$ in \eqref{proof_so_far_1000} by its upper bound in \eqref{sth_needed} yields
\begin{align}\label{proof_so_far_3000}
&F(\bbz_{k+1})^{\top}(\bbz_{k+1} - \bbz) \nonumber\\
& \leq \frac{1}{2\eta} \| \bbz_{k} - \bbz \|^2 - \frac{1}{2\eta} \| \bbz_{k+1} - \bbz \|^2 - \frac{1}{2\eta} \| \bbz_{k+1} - \bbz_{k} \|^2 \nonumber \\
& \qquad +(F(\bbz_{k+1}) - F(\bbz_k))^{\top}(\bbz_{k+1} - \bbz) - (F(\bbz_{k}) - F(\bbz_{k-1}))^{\top}(\bbz_{k} - \bbz) \nonumber \\
& \qquad +  \frac{L}{2} \| \bbz_{k} - \bbz_{k-1}\|^2 +  \frac{L}{2} \|\bbz_{k+1} - \bbz_k\|^2\nonumber \\
& \leq \frac{1}{2\eta} \| \bbz_{k} - \bbz\|^2 - \frac{1}{2\eta} \| \bbz_{k+1} - \bbz\|^2 - \frac{L}{2} \| \bbz_{k+1} - \bbz_{k} \|^2 +  \frac{L}{2} \| \bbz_{k} - \bbz_{k-1}\|^2 \nonumber \\
& \qquad +(F(\bbz_{k+1}) - F(\bbz_k))^{\top}(\bbz_{k+1} - \bbz) - (F(\bbz_{k}) - F(\bbz_{k-1}))^{\top}(\bbz_{k} - \bbz),
\end{align}
where the second inequality follows as $\eta\leq 1/2L$ and therefore $- \frac{1}{2\eta} \| \bbz_{k+1} - \bbz_{k} \|^2\leq -{L} \| \bbz_{k+1} - \bbz_{k} \|^2$. This completes the proof of Part (a) of the lemma. Now, taking the sum of the preceding relation from $k = 0, \cdots, N-1$, we obtain
\begin{align}\label{proof_so_far_4000}
&\sum_{k=0}^{N-1} F(\bbz_{k+1})^{\top}(\bbz_{k+1} - \bbz) \nonumber\\
& \leq \frac{1}{2\eta} \| \bbz_{0} - \bbz \|^2 - \frac{1}{2\eta} \| \bbz_{N} - \bbz \|^2 - \frac{L}{2} \| \bbz_{N} - \bbz_{N-1} \|^2+ \frac{L}{2} \| \bbz_{0} - \bbz_{-1} \|^2 \nonumber \\
& \qquad +(F(\bbz_{N}) - F(\bbz_{N-1}))^{\top}(\bbz_{N} - \bbz) - (F(\bbz_{0}) - F(\bbz_{-1}))^{\top}(\bbz_{0} - \bbz).
\end{align}
Now set $\bbz = \bbz^*$, where $\bbz^* \in \mathcal{Z}^*$, to obtain 
\begin{align}\label{proof_so_far_4010}
&\sum_{k=0}^{N-1} F(\bbz_{k+1})^{\top}(\bbz_{k+1} - \bbz^*) \nonumber\\
& \leq \frac{1}{2\eta} \| \bbz_{0} - \bbz^* \|^2 - \frac{1}{2\eta} \| \bbz_{N} - \bbz^* \|^2 - \frac{L}{2} \| \bbz_{N} - \bbz_{N-1} \|^2+ \frac{L}{2} \| \bbz_{0} - \bbz_{-1} \|^2 \nonumber \\
& \qquad +(F(\bbz_{N}) - F(\bbz_{N-1}))^{\top}(\bbz_{N} - \bbz^*) - (F(\bbz_{0}) - F(\bbz_{-1}))^{\top}(\bbz_{0} - \bbz^*).
\end{align}
Note that each term of the summand in the sum in the left is nonnegative due to monotonicity of $F$ and therefore the sum is also nonnegative. Further, we know that $\bbz_0 = \bbz_{-1}$. Using these observations we can write 
\begin{align}\label{proof_so_far_5000}
0& \leq \frac{1}{2\eta} \| \bbz_{0} - \bbz^* \|^2 - \frac{1}{2\eta} \| \bbz_{N} - \bbz^* \|^2 - \frac{L}{2} \| \bbz_{N} - \bbz_{N-1} \|^2 \nonumber \\
& \qquad +(F(\bbz_{N}) - F(\bbz_{N-1}))^{\top}(\bbz_{N} - \bbz^*).
\end{align}
Using Lipschitz continuity of the operator $F(\cdot)$ (Lemma \ref{lemma:operator_monotone}(b)) and Young's inequality in the preceding relation, we have
\begin{align}\label{proof_so_far_6000}
0& \leq \frac{1}{2\eta} \| \bbz_{0} - \bbz^* \|^2 - \frac{1}{2\eta} \| \bbz_{N} - \bbz^* \|^2 - \frac{L}{2} \| \bbz_{N} - \bbz_{N-1} \|^2 \nonumber \\
& \qquad + L \| \bbz_{N}- \bbz_{N-1} \| \| \bbz_{N} - \bbz^* \|
 \nonumber \\
&\leq  \frac{1}{2\eta} \| \bbz_{0} - \bbz^* \|^2 - \frac{1}{2\eta} \| \bbz_{N} - \bbz^* \|^2 - \frac{L}{2} \| \bbz_{N} - \bbz_{N-1} \|^2 \nonumber \\
& \qquad +\frac{L}{2}\|\bbz_{N}- \bbz_{N-1}\|^{2}+\frac{L}{2}\|\bbz_{N} - \bbz^*\|^2
 \nonumber \\
& \leq \frac{1}{2\eta} \| \bbz_{0} - \bbz^* \|^2 - \frac{1}{2\eta} \| \bbz_{N} - \bbz^* \|^2 +\frac{L}{2}\|\bbz_{N} - \bbz^*\|^2
\end{align}
Regrouping the terms gives us
\begin{align}\label{proof_so_far_7000}
\|\bbz_{N} - \bbz^*\|^2\leq \frac{1}{(1 - \eta L)} \| \bbz_{0} - \bbz^* \|^2 .
\end{align}
Using the condition $\eta\leq 1/2L$, it follows that for any iterate $N$ we have 
\begin{align}\label{proof_so_far_8000}
\|\bbz_{N} - \bbz^*\|^2\leq 2 \| \bbz_{0} - \bbz^* \|^2 ,
\end{align}
%
and the claim in Part (b) follows.
\end{proof}

According to Lemma \ref{lemma:ogda_bounded_set}, the sequence of iterates $\{\bbx_k,\bby_k\}$ generated by OGDA method stays within a closed and bounded convex set. We use this result to prove a sublinear convergence rate of $\mathcal{O}(1/k)$ for the function value of the averaged iterates generated by OGDA to the function value at a saddle point, for smooth and convex-concave functions in the following theorem.

\begin{theorem}
\label{thm_OGDA_conv}
Suppose Assumptions \ref{ass:convex_concave},~\ref{ass:scsc} and \ref{ass:sol_set} hold. Let $\{ \bbx_k, \bby_k \}$ be the iterates generated by the OGDA updates in \eqref{OGDA_update}. Let the initial conditions satisfy $\bbx_0 = \bbx_{-1}$ and $\bby_0 = \bby_{-1}$. Consider the definition of the averaged iterates $\hbx_N,\hby_N$ in \eqref{avg_iterate} and the compact convex set $\ccalD$ in \eqref{set_D}.  If the stepsize $\eta$ satisfies the condition $0 < \eta\leq 1/2L$, then for all $N \geq 1$, we have
\begin{align}\label{OGDA_main_claim}
\left[\max_{\bby: (\hat{\bbx}_N, \bby) \in \mathcal{D}} f(\hat{\bbx}_N, \bby) - f^{\star}\right] + \left[f^{\star} - \min_{\bbx: (\bbx, \hat{\bby}_N) \in \mathcal{D}} f(\bbx, \hat{\bby}_N)\right] \leq \frac{D(8L + \frac{1}{2\eta})}{N},
\end{align}
where $f^{\star} = f(\bbx^{*}, \bby^{*})$
and $D=\|\bbx_0-\bbx^*\|^2+\|\bby_0-\bby^*\|^2$.
\end{theorem}

\begin{proof}
From Lemma \ref{lemma:ogda_bounded_set}(a), we have that the iterates generated by the OGDA method satisfy Equation  \eqref{proof_so_far_3000_1}. On taking the sum of this relation from $k = 0, \cdots, N-1$, we obtain for any $\bbz$
\begin{align}\label{proof_so_far_2_1000}
&\sum_{k=0}^{N-1} F(\bbz_{k+1})^{\top}(\bbz_{k+1} - \bbz) \nonumber\\
& \leq \frac{1}{2\eta} \| \bbz_{0} - \bbz\|^2 - \frac{1}{2\eta} \| \bbz_{N} - \bbz \|^2 - \frac{L}{2} \| \bbz_{N} - \bbz_{N-1} \|^2+ \frac{L}{2} \| \bbz_{0} - \bbz_{-1} \|^2 \nonumber \\
& \qquad +(F(\bbz_{N}) - F(\bbz_{N-1}))^{\top}(\bbz_{N} - \bbz) - (F(\bbz_{0}) - F(\bbz_{-1}))^{\top}(\bbz_{0} - \bbz).
\end{align}
Note that for any $\bbz_1, \bbz_2 \in \mathcal{D}$, we have:
\begin{align}
\| \bbz_1 - \bbz_2\|^2  &\leq 2\|\bbz_1 - \bbz^*\|^2 + 2\|\bbz_2 - \bbz^*\|^2 \nonumber \\
&\leq 4\|\bbz_0 - \bbz^*\|^2 + 4\|\bbz_0 - \bbz^*\|^2 \nonumber \\
&\leq 8D
\end{align}
where we have used the fact that $\| \bbz - \bbz^*\|^2 \leq 2 \|\bbz_0 - \bbz^*\|^2 \ \forall \ \bbz \in \mathcal{D}$ along with the fact that $\forall \ \bba, \bbb \in \mathbb{R}^d, \|\bba + \bbb\|^2 \leq 2\| \bba\|^2 + 2\| \bbb \|^2$. As $\bbz_{-1}=\bbz_{0}$ and $\eta\leq 1/2L$, for any $\bbz\in \ccalD$ we have
\begin{align}
\frac{1}{N} \sum_{k = 0}^{N-1} F(\bbz_{k+1}) ^{\top}  (\bbz_{k+1} - \bbz) &\leq \frac{ \frac{1}{2\eta} \| \bbz_0 - \bbz \|^2 + (F(\bbz_{N}) - F(\bbz_{N-1}))^{\top}(\bbz_{N} - \bbz)}{N} \nonumber \\
&\leq \frac{D(8L + \frac{1}{2\eta})}{N}.
\label{eq:rate_1}
\end{align}
This inequality follows since:
\begin{align}
(F(\bbz_{N}) - F(\bbz_{N-1}))^{\top}(\bbz_{N} - \bbz) &\leq \| F(\bbz_{N}) - F(\bbz_{N-1}) \| \| \bbz_{N} - \bbz\| \nonumber \\
& \leq L \|\bbz_{N} - \bbz_{N-1}\| \| \bbz_{N} - \bbz\| 
\end{align}
and for any $\bbx, \bby \in \mathcal{D}$, we have:
\begin{align}
\|\bbx - \bby\| &\leq \|\bbx - \bbz^*\| + \| \bby - \bbz^* \| \nonumber \\
&\leq 2\sqrt{2D}
\end{align}
Therefore, we have:
\begin{align}
(F(\bbz_{N}) - F(\bbz_{N-1}))^{\top}(\bbz_{N} - \bbz) \leq 8LD
\end{align}
which immediately gives us Inequality \eqref{eq:rate_1}.
Combining relation \eqref{eq:rate_1} with Proposition~\ref{lemma:iter_erg_ineq}
we have that for all $\bbx, \bby \in \mathcal{D}$
\begin{align}
f(\hat{\bbx}_N, \bby) - f(\bbx, \hat{\bby}_N) \leq \frac{D(8L + \frac{1}{2\eta})}{N}.
\end{align}
which gives us the following convergence rate estimate:
\begin{align}
\left[\max_{\bby: (\hat{\bbx}_N, \bby) \in \mathcal{D}} f(\hat{\bbx}_N, \bby) - f^{\star}\right] + \left[f^{\star} - \min_{\bbx: (\bbx, \hat{\bby}_N) \in \mathcal{D}} f(\bbx, \hat{\bby}_N)\right] \leq \frac{D(8L + \frac{1}{2\eta})}{N}, \nonumber
\end{align}
where $f^{\star} = f(\bbx^*, \bby^*)$.
\end{proof}

Note that convergence in Theorem \ref{thm_OGDA_conv} is shown in terms of the Primal-Dual gap $\max_{\bby: (\hat{\bbx}_N, \bby) \in \mathcal{D}} f(\hat{\bbx}_N, \bby) - \min_{\bbx: (\bbx, \hat{\bby}_N) \in \mathcal{D}} f(\bbx, \hat{\bby}_N)$ which is a common measure to capture closeness to the solution in convex-concave setting (see \citep{nemirovski2004prox}). Indeed, the duality gap is zero if and only if $(\hat{\bbx}_N, \hat{\bby}_N)$ is a saddle point of the problem. The primal-dual gap also has the following game theoretic interpretation. If Player $\bbx$ is playing $\hat{\bbx}_N$, then $\max_{\bby: (\hat{\bbx}_N, \bby) \in \mathcal{D}} f(\hat{\bbx}_N, \bby)$ quantifies how much Player $\bby$ can gain by playing an action in the set $\mathcal{D}$. Similarly, if Player $\bby$ is playing $\hat{\bby}_N$, then $-\min_{\bbx: (\bbx, \hat{\bby}_N) \in \mathcal{D}} f(\bbx, \hat{\bby}_N)$ quantifies how much Player $\bbx$ can gain by playing an action in $\mathcal{D}$. Therefore, the quantity $\max_{\bby: (\hat{\bbx}_N, \bby) \in \mathcal{D}}f(\hat{\bbx}_N, \bby) - \min_{\bbx: (\bbx, \hat{\bby}_N) \in \mathcal{D}} f(\bbx, \hat{\bby}_N)$ is a measure of the sum of how much each player can gain if they unilaterally deviate from the strategy $(\hat{\bbx}_N, \hat{\bby}_N)$. This goes to zero at the Nash Equilibrium (saddle point), where no player can gain by unilaterally deviating from the equilibrium strategy.

Also, note that the result in Theorem \ref{thm_OGDA_conv} also implies that $|f(\hbx_N,\hby_N)-f^*|\leq 9LD/N$ as we show in the following corollary. 

\begin{corollary}\label{OGDA_cor}
Suppose Assumptions \ref{ass:convex_concave},~\ref{ass:scsc} and \ref{ass:sol_set} hold. Let $\{ \bbx_k, \bby_k \}$ be the iterates generated by the OGDA updates in \eqref{OGDA_update}. Consider the definition of the averaged iterates $\hbx_N,\hby_N$ in \eqref{avg_iterate}.  If the stepsize $\eta$ satisfies the condition $0 < \eta\leq 1/2L$, then for all $N \geq 1$, we have
$$|f(\hbx_N,\hby_N)-f^{\star}|\leq \frac{D(8L + \frac{1}{2\eta})}{N},$$
where $f^{\star} = f(\bbx^*, \bby^*)$.
\end{corollary}

\begin{proof}
Note that $[\max_{\bby: (\hat{\bbx}_N, \bby) \in \mathcal{D}} f(\hat{\bbx}_N, \bby) - f^{\star}]$  and $ [f^{\star} - \min_{\bbx: (\bbx, \hat{\bby}_N) \in \mathcal{D}} f(\bbx, \hat{\bby}_N)]$ are both nonnegative. To verify note that $$\max_{\bby: (\hat{\bbx}_N, \bby) \in \mathcal{D}} f(\hat{\bbx}_N, \bby)\geq f(\hat{\bbx}_N, \bby^*)\geq  f(\bbx^*, \bby^*)$$ and $$\min_{\bbx: (\bbx, \hat{\bby}_N) \in \mathcal{D}} f(\bbx, \hat{\bby}_N)\leq f(\bbx^*, \hat{\bby}_N)\leq f(\bbx^*, \bby^*)$$ (since $(\bbx^*, \bby^*) \in \mathcal{D}$). Further, note that $(\hbx_N,\hby_N)$ belongs to the set $\ccalD$. Hence, it yields $$f(\hbx_N,\hby_N)-f^{\star}\leq \max_{\bby: (\hat{\bbx}_N, \bby) \in \mathcal{D}} f(\hat{\bbx}_N, \bby) - f^{\star}\leq \frac{D(8L + \frac{1}{2\eta})}{N}.$$Also, we can show that $$f^{\star} -f(\hbx_N,\hby_N)\leq f^{\star} - \min_{\bbx: (\bbx, \hat{\bby}_N) \in \mathcal{D}} f(\bbx, \hat{\bby}_N)\leq \frac{D(8L + \frac{1}{2\eta})}{N}.$$Therefore, $|f(\hbx_N,\hby_N)-f^{\star}|\leq \frac{D(8L + \frac{1}{2\eta})}{N}$.
\end{proof}

The result in Corollary \ref{OGDA_cor} shows that the function value of the averaged iterates generated by OGDA converges to the function value at a saddle point of problem \eqref{main_prob} at a sublinear rate of $\mathcal{O}(1/k)$ when the function is smooth and convex-concave. To the best of our knowledge, this is the first non-asymptotic complexity bound for OGDA for the convex-concave setting. Moreover, note that without computing any extra gradient evaluation, i.e., computing only one gradient per iteration with respect to $\bbx$ and $\bby$, OGDA recovers the convergence rate of proximal point method. 


\section{Extragradient Method}
\label{sec:EG}

In this section, we consider finding a saddle point of a general smooth convex-concave function using the Extra-gradient (EG) method. Similar to our analysis of the OGDA method, we show that by interpreting the EG method as an approximation of the proximal point method it is possible to establish a convergence rate of $\mathcal{O}(1/k)$ through a simple analysis.

Consider the update of EG in which we first compute a set of mid-point iterates $\{\bbx_{k+\frac{1}{2}}, \bby_{k+\frac{1}{2}}\}$ using the gradients with respect to $\bbx$ and $\bby$ at the current iterate
\begin{align}\label{EG_first_step_update}
&\bbx_{k+\frac{1}{2}} = \bbx_{k} -\eta \nabla_\bbx f(\bbx_k,\bby_k),\nonumber\\
&\bby_{k+\frac{1}{2}} = \bby_{k} +\eta \nabla_\bby f(\bbx_k,\bby_k).
\end{align}
Then, we compute the next iterates of the EG method $\{\bbx_{k+1}, \bby_{k+1}\}$ using the gradients at the mid-points $\{\bbx_{k+\frac{1}{2}}, \bby_{k+\frac{1}{2}}\}$, i.e., 
\begin{align}\label{EG_second_step_update}
&\bbx_{k+1} = \bbx_{k} -\eta \nabla_\bbx f(\bbx_{k+\frac{1}{2}}, \bby_{k+\frac{1}{2}}),\nonumber\\
&\bby_{k+1} = \bby_{k} +\eta \nabla_\bby f(\bbx_{k+\frac{1}{2}}, \bby_{k+\frac{1}{2}}).
\end{align}
We aim to show that EG, similar to OGDA, can be analyzed for convex-concave problems by considering it as an approximation of the proximal point. To do so, let us use the notation  $\bbz = [\bbx; \bby]\in \reals^{n+m}$ and  $F(\bbz) = [\nabla_{\bbx} f(\bbx, \bby) ;  - \nabla_{\bby} f(\bbx, \bby)  ]\in \reals^{n+m}$ to write the update of EG as
\begin{align}\label{VI_EG_update}
\bbz_{k+\frac{1}{2}} &= \bbz_k - \eta F(\bbz_k), \nonumber \\
\bbz_{k+1} &= \bbz_k - \eta F(\bbz_{k+\frac{1}{2}}).
\end{align}
To better highlight the connection between proximal point and EG, let us focus on the expression for the update of the mid-point iterates in EG. Considering the updates in \eqref{VI_EG_update}, we have
\begin{align}
\bbz_{k+\frac{1}{2}} &= \bbz_k - \eta F(\bbz_k), \nonumber \\
&= \bbz_{k-1} - \eta F(\bbz_{k-\frac{1}{2}}) - \eta F(\bbz_k) \nonumber \\
&= \bbz_{k-\frac{1}{2}} + \eta F(\bbz_{k-1}) - \eta F(\bbz_{k-\frac{1}{2}}) - \eta F(\bbz_k) \nonumber
\end{align}
where the second equality follows by replacing $\bbz_k$ by its update $ \bbz_{k-1} - \eta F(\bbz_{k-\frac{1}{2}}) $, and the second equality follows by considering the update $\bbz_{k-\frac{1}{2}}=\bbz_{k-1}-\eta F(\bbz_{k-1})$. Therefore, rearranging this equation, we can rewrite the updates as
\begin{equation}\label{EG_VI_mid_update}
\bbz_{k+\frac{1}{2}} = \bbz_{k-\frac{1}{2}} - \eta F(\bbz_{k-\frac{1}{2}}) - \eta(F(\bbz_k) - F(\bbz_{k-1})).
\end{equation}
One can consider the expression $F(\bbz_k) - F(\bbz_{k-1})$ as an approximation of the variation $F(\bbz_{k+\frac{1}{2}}) -F(\bbz_{k-\frac{1}{2}})$. To be more precise, if we assume that the variations $F(\bbz_k) - F(\bbz_{k-1})$ and $F(\bbz_{k+\frac{1}{2}}) -F(\bbz_{k-\frac{1}{2}})$ are close to each other, i.e., $F(\bbz_{k+\frac{1}{2}}) -F(\bbz_{k-\frac{1}{2}})\approx F(\bbz_k) - F(\bbz_{k-1})$, then the update in \eqref{EG_VI_mid_update} behaves like the proximal point update with. respect to the mid-point iterates, i.e., 
\begin{equation}\label{prox_VI_update_2}
\bbz_{k+\frac{1}{2}} = \bbz_{k-\frac{1}{2}} - \eta F(\bbz_{k+\frac{1}{2}}).
\end{equation}
We first derive a result similar to Proposition~\ref{lemma_general_pp_approx} for the specific case of EG iterates (Lemma~\ref{lemma:bdd_iter_EG}(a)). We then show the the boundedness of the EG iterates in Lemma~\ref{lemma:bdd_iter_EG}(b) (Note that the boundedness of the EG updates can also be deduced from the convergence results of  \citet{korpelevich1976extragradient} and \cite{monteiro2010complexity}).
\begin{lemma}\label{lemma:bdd_iter_EG}
Let $\{\bbz_k\},\{ \bbz_{k+\frac{1}{2}} \}$ be the iterates generated by the extra-gradient (EG) method introduced in \eqref{VI_EG_update}. If Assumptions~\ref{ass:convex_concave},~\ref{ass:scsc} and \ref{ass:sol_set} hold and the stepsize $\eta$ satisfies the condition $0 < \eta < {1}/{L}$, then: \\
(a) The iterates $\{\bbz_k\},\{ \bbz_{k+\frac{1}{2}} \}$ satisfy the following relation:
\begin{align}\label{eq:before_sum_EG_1}
F(\bbz_{k+\frac{1}{2}}) ^{\top}  & (\bbz_{k+\frac{1}{2}} - \bbz) \nonumber  \\
&\leq \frac{1}{2\eta}\|\bbz_{k-\frac{1}{2}} - \bbz\|^2 -  \frac{1}{2\eta} \|\bbz_{k+\frac{1}{2}} - \bbz\|^2+  \frac{L}{2} \| \bbz_{k-\frac{1}{2}} - \bbz_{k-1}\|^2 \nonumber \\
& \quad +(F(\bbz_{k+\frac{1}{2}}) - F(\bbz_k))^{\top}(\bbz_{k+\frac{1}{2}} - \bbz) - (F(\bbz_{k-\frac{1}{2}}) - F(\bbz_{k-1}))^{\top}(\bbz_{k-\frac{1}{2}} - \bbz).
\end{align}
(b) The iterates $\{\bbz_k\},\{ \bbz_{k+\frac{1}{2}} \}$ stay within the compact set $\ccalD$ defined as 
\begin{equation}\label{bounded_set_eg}
\ccalD:=\{(\bbx,\bby)\mid \| \bbx - \bbx^* \|^2 + \| \bby - \bby^* \|^2\leq  \left(2 + \frac{2}{1 - \eta^2 L^2} \right) (\| \bbx_0 - \bbx^* \|^2 + \| \bby_0 - \bby^* \|^2)  \},
\end{equation}
where $(\bbx^*,\bby^*) = \bbz^* \in \mathcal{Z}^*$ is a saddle point of the problem defined in \eqref{main_prob}. Moreover, the sum $\sum_{k=0}^{\infty} \| \bbz_{k+\frac{1}{2}} - \bbz_k \|^2$ is bounded above by
\begin{align}\label{bounded_sum_claim_EG}
\sum_{k=0}^{\infty} \| \bbz_{k+\frac{1}{2}} - \bbz_k \|^2 \leq \frac{\|\bbz_{0} - \bbz^*\|^2} {1 - \eta^2 L^2}.
\end{align}
\end{lemma}

\begin{proof}
See Appendix \ref{appx:proof_of_lemma_EG} for the proof.
\end{proof}


The result in Lemma~\ref{lemma:bdd_iter_EG} shows that the iterates generated by the update of EG belong to a bounded and closed set. 
Now we use this result to show that the function value of the averaged iterates converges at a sublinear rate of $\mathcal{O}(1/k)$  to the function value at a saddle point for the EG method in the following theorem.


\begin{theorem}
\label{thm_EG_convergence}
Suppose Assumptions \ref{ass:convex_concave},~\ref{ass:scsc} and \ref{ass:sol_set} hold. Let $\{ \bbx_{k+1/2}, \bby_{k+1/2} \}$ be the iterates generated by the EG updates in  \eqref{EG_first_step_update}-\eqref{EG_second_step_update}. 
Let the initial conditions satisfy $\bbx_0 = \bbx_{-1/2}$ and $\bby_0 = \bby_{-1/2}$
Consider the definition of the averaged iterates $\hbx_N,\hby_N$ in \eqref{avg_iterate} and the compact convex set $\ccalD$ in \eqref{bounded_set_eg}.  If the stepsize $\eta$
satisfies the condition $\eta =  \frac{\sigma}{L}$ for any $\sigma \in (0,1)$, then for all $N \geq 1$, we have
\begin{align}\label{EG_main_claim}
\left[\max_{\bby: (\hat{\bbx}_N, \bby) \in \mathcal{D}} f(\hat{\bbx}_N, \bby) - f^{\star}\right] + \left[f^{\star} - \min_{\bbx: (\bbx, \hat{\bby}_N) \in \mathcal{D}} f(\bbx, \hat{\bby}_N)\right] \leq \frac{DL \left( 16 + \frac{33}{2(1-\sigma^2)} \right)}{N},
\end{align}
where $f^{\star} =  f(\bbx^*, \bby^*)$
and $D=\|\bbx_0-\bbx^*\|^2+\|\bby_0-\bby^*\|^2$.
\end{theorem}

\begin{proof}
See Appendix \ref{proof_thm_EG} for the proof.
\end{proof}
Now, similar to Corollary \ref{OGDA_cor}, we have:
\begin{corollary}\label{EG_cor}
Suppose Assumptions \ref{ass:convex_concave},~\ref{ass:scsc} and \ref{ass:sol_set} hold. Let $\{ \bbx_{k+1/2}, \bby_{k+1/2} \}$ be the iterates generated by the EG updates in  \eqref{EG_first_step_update}-\eqref{EG_second_step_update}. 
Let the initial conditions satisfy $\bbx_0 = \bbx_{-1/2}$ and $\bby_0 = \bby_{-1/2}$
Consider the definition of the averaged iterates $\hbx_N,\hby_N$ in \eqref{avg_iterate}. If the stepsize $\eta$
satisfies the condition $\eta =  \frac{\sigma}{L}$ for any $\sigma \in (0,1)$, then for all $N \geq 1$, we have
$$|f(\hbx_N,\hby_N)-f^{\star}|\leq \frac{DL \left( 16 + \frac{33}{2(1-\sigma^2)} \right)}{N},$$
where $f^{\star} = f(\bbx^*, \bby^*)$.
\end{corollary}




\section{Discussion and Numerical Experiments}
The main message of this work is that the OGDA algorithm obtains the same convergence rate of $\mathcal{O}(1/k)$, the best achievable rate (see \citep{nemirovski2004prox}), also achieved by EG. However, the advantage of OGDA is that we need only one gradient computation at each step, as opposed to two gradient computations needed in EG. This shows the computational advantage that OGDA has over EG. 

\begin{figure}[t!]
  \centering
\includegraphics[width=0.8\columnwidth]{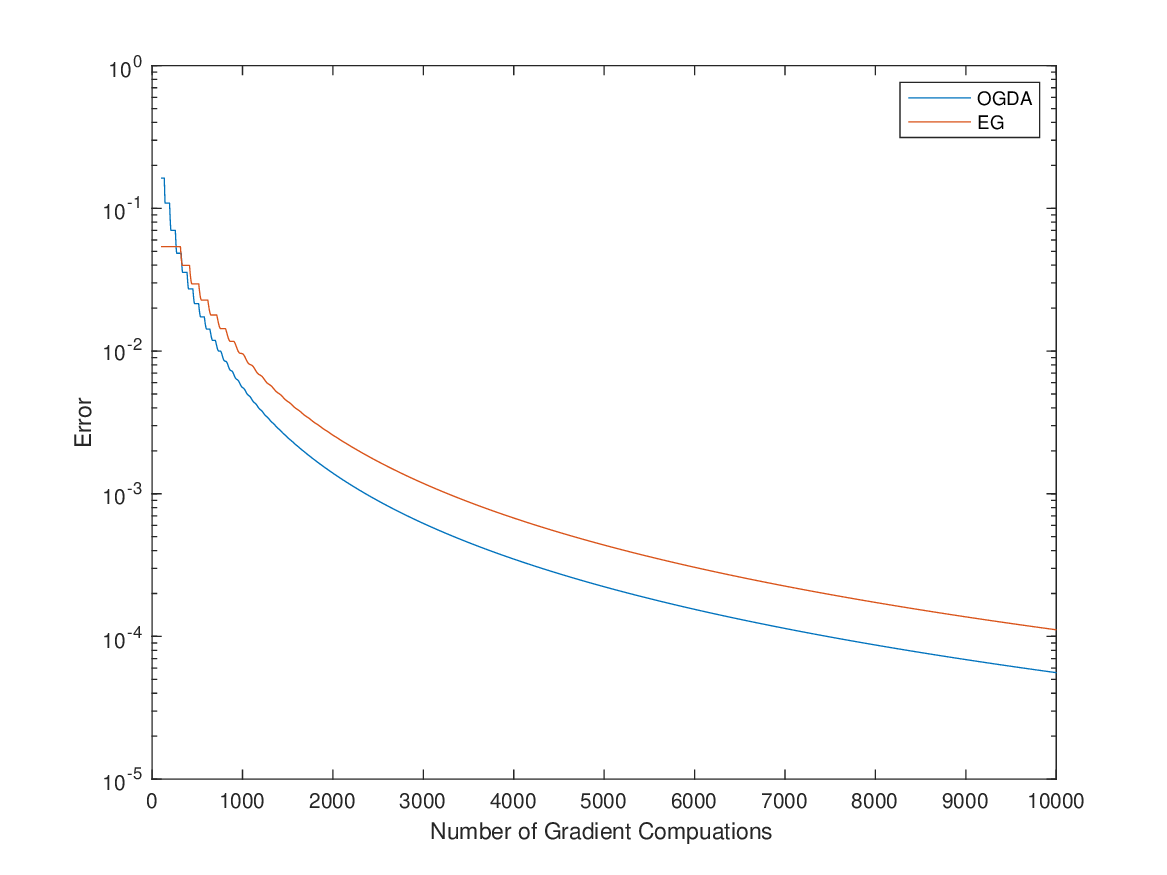}
\vspace{-2mm}
    \caption{Number of Gradient computations required ($x$-axis) to reach any error level ($y$-axis) for both OGDA and EG for the problem in Equation \eqref{eq:simul_bi}
}\vspace{-2mm}
 \label{fig_trajec}
 \end{figure}

We compare the performance of OGDA and EG in terms of gradient computations, on the bilinear minimax games considered in \citep{nemirovski2004prox}, without any constraint. In particular, we consider the following minimax problem:
\begin{align}
\label{eq:simul_bi}
\min_{\bbx \in \mathbb{R}^n} \max_{\bby \in \mathbb{R}^n} \ \bbx^{\top} \bbB \bby,
\end{align}
where $\bbB \in \mathbb{R}^{n \times n}$ is a sparse random matrix generated as follows. Each element is nonzero independently with probability $p$. If an element is chosen to be non-zero, it is chosen uniformly from $[-1,1]$. We compare the number of gradient computations required to reach a desired accuracy level for this problem in Figure \ref{fig_trajec}. As we observe, both EG and OGDA converge to the saddle point of the bilinear problem at a sublinear rate of $\mathcal{O}(1/k)$, but OGDA slightly outperforms EG in terms of number of gradient evaluations. Once again, this is due to the fact that for both descent and ascent updates of OGDA requires only one gradient computation each, while EG requires two gradient computations for both updates at each iteration. 

Note that the Lipschitz constants for the considered problem can be estimated from data using standard line search techniques. In particular, \cite{beck2009fast} discuss a backward tracking algorithm (ISTA with backtracking) which can be used to estimate the Lipschitz constants, in particular $L_{xx}$ and $L_{yy}$. Several variants of this algorithm, including the Lipschitz line-search algorithm (Algorithm 2) in \citep{schmidt2015non}, can also be used to estimate the Lipschitz constants $L_{xx}$ and $L_{yy}$. For the specific case of saddle point problems, a recent paper \citep{hamedani2018primal} proposes a line search algorithm, to estimate the Lipschitz constant $L_{xx}, L_{xy},L_{yx}$ and  $L_{yy}$. They propose an algorithm - Accelerated Primal Dual with backtracking (Algorithm 2.3) which uses a backtracking procedure, similar to  \citep{malitsky2018first}, to locally estimate the Lipschitz constants of the problem. Also, regarding the initial error, we would like to highlight that in the analysis of convex minimization problems or convex-concave saddle point problems, we often have a term of the form $\|\bbx_0-\bbx^*\|^2$ in the upper bound (for instance see \citep{nesterov2013introductory, monteiro2010complexity}) which shows the effect of initial error. This parameter is hard to estimate in general but can be upper bounded in specific cases. For example, if we are looking at for mixed strategies in zero-sum games, we know that we are looking for a solution lies in the probability simplex, so we can bound the initial error simply by the diameter of the simplex. In general, if we know that our iterates of the algorithm are going to lie in some compact set, we can upper bound the initial distance to the solution simply by the diameter of the compact set.

\section{Conclusions}
In this paper, we established convergence guarantees of the optimistic gradient ascent-descent (OGDA) and Extra-gradient (EG) methods for unconstrained, smooth, and convex-concave saddle point problems. In particular, we showed a sublinear convergence rate of $\mathcal{O}(1/k)$ in terms of function value error for both OGDA and EG by interpreting them as approximate variants of the proximal point method. This result leads to the first theoretical guarantee for OGDA in convex-concave saddle point problems. Moreover, it provides a simple and short proof for the convergence rate of EG in convex-concave saddle point problems when we measure optimality gap in terms of function value.

%
%





\appendix


\section{Proof of Theorem~\ref{theorem:prox_point_rate}}\label{appx:proof_of_thm_pp}

\hfill \newline
The update of the proximal point method can be written as:
\begin{align}
\bbz_{k+1} = \bbz_k - \eta F(\bbz_{k+1} )
\end{align}
According to this update we can show that
\begin{align}
\|\bbz_{k+1} - \bbz\|^2 = \|\bbz_k - \bbz\|^2 -  2\eta (\bbz_k - \bbz)^\top  F(\bbz_{k+1})  & + \eta^2 \|F(\bbz_{k+1})  \|^2 
\end{align}
Now add and subtract the inner product $2\eta\bbz_{k+1}^\top  F(\bbz_{k+1})$ to the right hand side and regroup the terms to obtain
\begin{align}
\|\bbz_{k+1} - \bbz\|^2 
&= \|\bbz_k - \bbz\|^2 - 2\eta (\bbz_{k+1} - \bbz)^{\top} F(\bbz_{k+1})- 2\eta (\bbx_{k} - \bbx_{k+1})^{\top}  F(\bbz_{k+1})\nonumber \\
&\quad + \eta^2 \| F(\bbz_{k+1}) \|^2.
\end{align}
Replace $F(\bbz_{k+1})$ with $(1/\eta)(-\bbz_{k+1}+\bbz_k)$ to obtain
\begin{align}
&\|\bbz_{k+1} - \bbz\|^2 \nonumber\\
&= \|\bbz_k - \bbz\|^2 - 2\eta (\bbz_{k+1} - \bbz)^{\top} F(\bbz_{k+1})+ 2(\bbz_{k} - \bbz_{k+1})^\top (\bbz_{k+1} - \bbz_k ) \nonumber \\
& \qquad + \|\bbz_{k+1} - \bbz_k  \|^2  \nonumber \\
&=  \|\bbz_k - \bbz\|^2 - 2\eta (\bbz_{k+1} - \bbz)^{\top} F(\bbz_{k+1})   - \|\bbz_{k+1} - \bbz_{k}\|^2 .
\end{align}
On rearranging the terms, we get the following
\begin{align}
&F(\bbz_{k+1}) ^{\top} (\bbz_{k+1} - \bbz) 
= \frac{1}{2\eta}  \|\bbz_k - \bbz\|^2 - \frac{1}{2\eta}  \|\bbz_{k+1} - \bbz\|^2 -\frac{1}{2\eta}  \|\bbz_{k+1} - \bbz_k\|^2,
\label{eq:to_add}
\end{align}
Now, on substituting $\bbz = \bbz^*$, and noting that $F(\bbz_{k+1}) ^{\top} (\bbz_{k+1} - \bbz^*) \geq 0$, we have:
\begin{align}
 \|\bbz_{k+1} - \bbz^*\|^2 \leq  \|\bbz_{k} - \bbz^*\|^2 -  \|\bbz_{k+1} - \bbz_k\|^2
\end{align}
and the proof of boundedness is complete.

On adding Equation \eqref{eq:to_add} from $k = 0, \cdots N-1$ and diving by $N$, we get:
\begin{align}
\frac{1}{N} \sum_{k=1}^{N} F(\bbz_{k}) ^{\top} (\bbz_{k} - \bbz) \leq \frac{\|\bbz_0 - \bbz\|^2}{\eta N}
\end{align}
Now, using Proposition~\ref{lemma:iter_erg_ineq} we can write
\begin{align}
| f(\hat{\bbx}_N, \hat{\bby}_N) - f^{\star} | \leq \frac{\|\bbx_0 - \bbx\|^2 + \|\bby_0 - \bby\|^2}{\eta N},
\end{align}
and the proof is complete.

\section{Proof of Lemma \ref{lemma:bdd_iter_EG}}\label{appx:proof_of_lemma_EG}

\hfill \newline
(a) Considering the updates in \eqref{EG_VI_mid_update} and \eqref{prox_VI_update_2} we can write the update of mid-points in EG as 
\begin{equation}\label{proof_EG_1000}
\bbz_{k+\frac{1}{2}} = \bbz_{k-\frac{1}{2}} - \eta F(\bbz_{k+\frac{1}{2}}) +\bbvarepsilon_k, 
\end{equation}
where,
\begin{equation}\label{proof_EG_2000}
\bbvarepsilon_k = \eta\left[(F(\bbz_{k+\frac{1}{2}}) -F(\bbz_{k-\frac{1}{2}}))- ( F(\bbz_k) - F(\bbz_{k-1}))\right].
\end{equation}
Therefore, we can simplify the last term in Equation \eqref{proof_mid_lemma_600} of Proposition~\ref{lemma_general_pp_approx} as follows:
\begin{align}   \label{eq:EG_almost_err} 
 \frac{1}{\eta} & {\bbvarepsilon_k}^{\top} (\bbz_{k+\frac{1}{2}} - \bbz) \nonumber \\
  &=  \frac{1}{\eta} \times[ (\eta F(\bbz_{k+\frac{1}{2}}) - \eta F(\bbz_k )) - (\eta F(\bbz_{k-\frac{1}{2}}) - \eta F(\bbz_{k-1} )) ]  ^{\top} (\bbz_{k+\frac{1}{2}} - \bbz) \nonumber \\
  &= (F(\bbz_{k+\frac{1}{2}}) - F(\bbz_k ))^{\top}(\bbz_{k+\frac{1}{2}} - \bbz)  - ( F(\bbz_{k-\frac{1}{2}}) -  F(\bbz_{k-1} ))^{\top} (\bbz_{k-\frac{1}{2}} - \bbz) \nonumber \\
  & \qquad \qquad \qquad \qquad - ( F(\bbz_{k-\frac{1}{2}}) -  F(\bbz_{k-1} ))^{\top} (\bbz_{k+\frac{1}{2}} - \bbz_{k-\frac{1}{2}}).
\end{align}
Using Lipschitz continuity of the operator $F$ (Lemma \ref{lemma:operator_monotone}(b)) and Young's inequality, we have 
\begin{align}  \label{eq:EG_almost_err_2} 
 \frac{1}{\eta}  {\bbvarepsilon_k}^{\top} & (\bbz_{k+\frac{1}{2}} - \bbz) \nonumber \\
  & \leq F(\bbz_{k+\frac{1}{2}}) - F(\bbz_k ))^{\top}(\bbz_{k+\frac{1}{2}} - \bbz)  - ( F(\bbz_{k-\frac{1}{2}}) -  F(\bbz_{k-1} ))^{\top} (\bbz_{k-\frac{1}{2}} - \bbz) \nonumber \\
  & \qquad \qquad \qquad \qquad +L \| \bbz_{k-\frac{1}{2}} - \bbz_{k-1} \| \| \bbz_{k+\frac{1}{2}} - \bbz_{k-\frac{1}{2}} \| \nonumber \\
  & \leq F(\bbz_{k+\frac{1}{2}}) - F(\bbz_k ))^{\top}(\bbz_{k+\frac{1}{2}} - \bbz)  - ( F(\bbz_{k-\frac{1}{2}}) -  F(\bbz_{k-1} ))^{\top} (\bbz_{k-\frac{1}{2}} - \bbz) \nonumber \\
  & \qquad \qquad \qquad \qquad +\frac{L}{2} \| \bbz_{k-\frac{1}{2}} - \bbz_{k-1} \|^2 + \frac{L}{2} \| \bbz_{k+\frac{1}{2}} - \bbz_{k-\frac{1}{2}} \|^2
\end{align}
Substituting the upper bound in \eqref{eq:EG_almost_err_2} into Equation \eqref{proof_mid_lemma_600} of Proposition~\ref{lemma_general_pp_approx}, implies that
\begin{align}
 F(\bbz_{k+\frac{1}{2}}) & ^{\top}  (\bbz_{k+\frac{1}{2}} - \bbz) \nonumber  \\
&\leq \frac{1}{2\eta}\|\bbz_{k-\frac{1}{2}} - \bbz\|^2 -  \frac{1}{2\eta} \|\bbz_{k+\frac{1}{2}} - \bbz\|^2 - \frac{1}{2\eta} \|\bbz_{k+\frac{1}{2}} - \bbz_{k-\frac{1}{2}}\|^2 \nonumber \\
& \qquad +(F(\bbz_{k+\frac{1}{2}}) - F(\bbz_k))^{\top}(\bbz_{k+\frac{1}{2}} - \bbz) - (F(\bbz_{k-\frac{1}{2}}) - F(\bbz_{k-1}))^{\top}(\bbz_{k-\frac{1}{2}} - \bbz) \nonumber \\
& \qquad +  \frac{L}{2} \| \bbz_{k-\frac{1}{2}} - \bbz_{k-1}\|^2 +  \frac{L}{2} \|\bbz_{k+\frac{1}{2}} - \bbz_{k-\frac{1}{2}}\|^2.
\end{align}
Since $\eta< 1/L$, we have $- \frac{1}{2\eta} \|\bbz_{k+\frac{1}{2}} - \bbz_{k-\frac{1}{2}}\|^2+  \frac{L}{2} \|\bbz_{k+\frac{1}{2}} - \bbz_{k-\frac{1}{2}}\|^2\leq 0$ and therefore
\begin{align}\label{eq:before_sum_EG}
&F(\bbz_{k+\frac{1}{2}}) ^{\top}  (\bbz_{k+\frac{1}{2}} - \bbz) \nonumber  \\
&\leq \frac{1}{2\eta}\|\bbz_{k-\frac{1}{2}} - \bbz\|^2 -  \frac{1}{2\eta} \|\bbz_{k+\frac{1}{2}} - \bbz\|^2+  \frac{L}{2} \| \bbz_{k-\frac{1}{2}} - \bbz_{k-1}\|^2 \nonumber \\
& \quad +(F(\bbz_{k+\frac{1}{2}}) - F(\bbz_k))^{\top}(\bbz_{k+\frac{1}{2}} - \bbz) - (F(\bbz_{k-\frac{1}{2}}) - F(\bbz_{k-1}))^{\top}(\bbz_{k-\frac{1}{2}} - \bbz).
\end{align}
which completes the proof of Part (a). \\
(b) Based on the update of EG in \eqref{VI_EG_update}, we can write
\begin{align}
& \|\bbz_{k} - \bbz\|^2 \nonumber \\
&= \|\bbz_{k} - \bbz_{k+1} + \bbz_{k+1} -  \bbz\|^2 \nonumber \\
&= \| \bbz_{k+1} - \bbz \|^2 + 2 ( \bbz - \bbz_{k+1})^\top ( \bbz_{k+1} - \bbz_k) + \| \bbz_{k+1} - \bbz_k \|^2 \nonumber \\
&= \| \bbz_{k+1} - \bbz \|^2 + 2 (\bbz - \bbz_{k+\frac{1}{2}})^\top( \bbz_{k+1} - \bbz_k )\nonumber\\
&\qquad + 2( \bbz_{k+\frac{1}{2}} - \bbz_{k+1})^\top ( \bbz_{k+1} - \bbz_k)+  \| \bbz_{k+1} - \bbz_k \|^2 \nonumber \\
&=  \| \bbz_{k+1} - \bbz \|^2 + 2 ( \bbz - \bbz_{k+\frac{1}{2}})^\top (\bbz_{k+1} - \bbz_k ) + \| \bbz_{k+\frac{1}{2}} - \bbz_k \|^2 - \| \bbz_{k+\frac{1}{2}} - \bbz_{k+1} \|^2.
\label{eq:EG_bdd_mid}
\end{align}
Now we proceed to bound the difference $\| \bbz_{k+\frac{1}{2}} - \bbz_{k+1} \|^2$. Using the fact that the operator $F$ is $L$-Lipschitz (Lemma \ref{lemma:operator_monotone}(b)), we have
\begin{align}
\| \bbz_{k+\frac{1}{2}} - \bbz_{k+1} \|^2 &= \eta^2 \| F(\bbz_{k+\frac{1}{2}} - F(\bbz_k) \|^2 \nonumber \\
&\leq \eta^2 L^2 \| \bbz_{k+\frac{1}{2}} - \bbz_{k} \|^2.
\end{align}
Substituting this upper bound back into \eqref{eq:EG_bdd_mid} and taking $\bbz = \bbz^*$ implies
\begin{align}\label{eq:EG_bdd_mid_2}
\|\bbz_{k} & - \bbz^*\|^2 \nonumber \\
&\geq \| \bbz_{k+1} - \bbz^* \|^2 + 2 (\bbz^* - \bbz_{k+\frac{1}{2}})^\top( \bbz_{k+1}  - \bbz_k)  + (1 - \eta^2 L^2) \| \bbz_{k+\frac{1}{2}} - \bbz_k \|^2.
\end{align}
Further, since the operator $F$ is monotone, we have
\begin{align}
(\bbz^* - \bbz_{k+\frac{1}{2}})^\top( \bbz_{k+1} - \bbz_k)  &=  \eta  (F(\bbz_{k+\frac{1}{2}}) )^\top(  \bbz_{k+\frac{1}{2}} - \bbz^*  ) \nonumber \\
&\geq \eta  ( F(\bbz_{k+\frac{1}{2}}) - F(\bbz^*))^\top( \bbz_{k+\frac{1}{2}} - \bbz^*) \nonumber \\
&\geq 0,
\end{align}
where in the first inequality we used the fact that $F(\bbz^*) = \bb0$ (Lemma \ref{lemma:operator_monotone}(c)), and the last inequality holds due to monotonicity of $F$ (Lemma \ref{lemma:operator_monotone}(a)). Therefore, we can replace the inner product $2(\bbz^* - \bbz_{k+\frac{1}{2}})^\top( \bbz_{k+1} - \bbz_k) $ in \eqref{eq:EG_bdd_mid_2} by its lower bound $0$ to obtain
\begin{align}\label{eq:EG_bdd_mid_222}
\|\bbz_{k} - \bbz^*\|^2 \geq \| \bbz_{k+1} - \bbz^* \|^2 + (1 - \eta^2 L^2) \| \bbz_{k+\frac{1}{2}} - \bbz_k \|^2
\end{align}
The result in \eqref{eq:EG_bdd_mid_222} shows that the seqeunce $\|\bbz_{k} - \bbz^*\|^2$ is non-increasing. Therefore, for any iterate $k$, it holds that 
\begin{align}\label{eq:mid-iter_bdd_EG}
\|\bbz_k - \bbz^* \|^2 \leq \|\bbz_0 - \bbz^* \|^2.
\end{align}
Now, for all $k \geq 0$, we have:
\begin{align}\label{eq:iter_bdd_EG_new_set}
\| \bbz_{k+\frac{1}{2}} - \bbz^* \|^2 &\leq 2 \| \bbz_{k} - \bbz^* \|^2 + 2 \| \bbz_{k+\frac{1}{2}} - \bbz_{k} \|^2 \nonumber \\
&\leq \left(2 + \frac{2}{1 - \eta^2 L^2} \right) \| \bbz_{k} - \bbz^* \|^2 \nonumber \\
&\leq \left(2 + \frac{2}{1 - \eta^2 L^2} \right) \| \bbz_{0} - \bbz^* \|^2
\end{align} 
where the first inequality follows from the fact that $\forall \ \bba, \bbb \in \mathbb{R}^d, \|\bba + \bbb\|^2 \leq 2\| \bba\|^2 + 2\| \bbb \|^2$, the second inequality follows from \eqref{eq:EG_bdd_mid_222} and the third inequality follows from \eqref{eq:mid-iter_bdd_EG}. Therefore from \eqref{eq:mid-iter_bdd_EG} and \eqref{eq:iter_bdd_EG_new_set}, since $0 < 1 - \eta^2L^2 < 1$, we see that the iterates $\{\bbz_k\},\{ \bbz_{k+\frac{1}{2}} \}$ belong to the compact set $\ccalD$ defined in \eqref{bounded_set_eg}.

Now by summing both sides of \eqref{eq:EG_bdd_mid_222} for $k = 0,\dots, \infty$, we obtain
\begin{align}
(1 - \eta^2 L^2) \sum_{k=0}^{\infty} \| \bbz_{k+\frac{1}{2}} - \bbz_k \|^2 \leq \|\bbz_{0} - \bbz^*\|^2
\end{align}
Therefore, by regrouping the terms we obtain 
\begin{align}\label{eq:summable}
\sum_{k=0}^{\infty} \| \bbz_{k+\frac{1}{2}} - \bbz_k \|^2 \leq \frac{\|\bbz_{0} - \bbz^*\|^2} {1 - \eta^2 L^2},
\end{align}
and the claim in \eqref{bounded_sum_claim_EG} follows.

\section{Proof of Theorem \ref{thm_EG_convergence}}\label{proof_thm_EG}

\hfill \newline
Using Equation \eqref{eq:before_sum_EG_1} of Lemma \ref{lemma:bdd_iter_EG}(a), summing it from $k = 0, \cdots , N-1$ and dividing by $N$, we obtain
\begin{align}
&\frac{1}{N} \sum_{k = 0}^{N-1} F(\bbz_{k+\frac{1}{2}}) ^{\top}  (\bbz_{k+\frac{1}{2}} - \bbz) \nonumber\\
& \leq \frac{\frac{1}{2\eta}\| \bbz_0\! -\! \bbz \|^2 + (F(\bbz_{N-\frac{1}{2}}) \!-\! F(\bbz_{N-1}))^{\top}(\bbz_{N-\frac{1}{2}}\!-\! \bbz)}{N}\!+\! \frac{L}{2N} \sum_{k=0}^{N-1} \| \bbz_{k-\frac{1}{2}} \!-\! \bbz_{k-1}\|^2.
\end{align}
The bound in Equation \eqref{bounded_sum_claim_EG} from Lemma \ref{lemma:bdd_iter_EG}(b) yields
\begin{align}
\frac{1}{N} \sum_{k = 0}^{N-1} & F(\bbz_{k+\frac{1}{2}}) ^{\top}  (\bbz_{k+\frac{1}{2}} - \bbz) \nonumber\\
&\leq \frac{L\| \bbz_0 - \bbz \|^2 + (F(\bbz_{N-\frac{1}{2}}) - F(\bbz_{N-1}))^{\top}(\bbz_{N-\frac{1}{2}} - \bbz)}{N} + \frac{L \| \bbz_{0} - \bbz^*\|^2}{2(1 - \eta^2 L^2)N}  \nonumber \\
&\leq \frac{L\| \bbz_0 - \bbz \|^2 +L\|\bbz_{N-\frac{1}{2}}-\bbz_{N-1}\|\|\bbz_{N-\frac{1}{2}} - \bbz\| +\frac{L}{2(1-\sigma^2)}\| \bbz_{0} - \bbz^*\|^2}{N},
\label{eq:rate_1_EG}
\end{align}
where in the last inequality we use Lipschitz continuity of the operator $F$ (Lemma \ref{lemma:operator_monotone}(b)) and the fact that $\eta = \frac{\sigma}{L}$. 
Note that for any $\bbz_1, \bbz_2 \in \mathcal{D}$, we have:
\begin{align}
\| \bbz_1 - \bbz_2\|  &\leq \|\bbz_1 - \bbz^*\| + \|\bbz_2 - \bbz^*\| \nonumber \\
&\leq \sqrt{\left(2 + \frac{2}{1 - \eta^2 L^2} \right)}\|\bbz_0 - \bbz^*\| + \sqrt{\left(2 + \frac{2}{1 - \eta^2 L^2} \right)}\|\bbz_0 - \bbz^*\| \nonumber \\
&\leq 2\sqrt{D \left(2 + \frac{2}{1 - \sigma^2} \right)}.
\end{align}
Therefore, for any point $\bbz$ in the set  $\ccalD$, we can substitute the preceding relation in Equation \eqref{eq:rate_1_EG} to get
\begin{align}
\frac{1}{N} \sum_{k = 0}^{N-1} F(\bbz_{k+\frac{1}{2}}) ^{\top}  (\bbz_{k+\frac{1}{2}} - \bbz) \leq \frac{DL \left( 16 + \frac{33}{2(1-\sigma^2)} \right) }{N}.
\end{align}
Now, using Proposition~\ref{lemma:iter_erg_ineq}
we have that for all $\bbx, \bby \in \mathcal{D}$:
\begin{align}
f(\hat{\bbx}_N, \bby) - f(\bbx, \hat{\bby}_N) \leq \frac{DL \left( 16 + \frac{33}{2(1-\sigma^2)} \right)}{N},
\end{align}
where $\hat{\bbx}_N = \frac{1}{N}  \sum_{k = 0}^{N-1} \bbx_{k+1/2}$ and $\hat{\bby}_N = \frac{1}{N}  \sum_{k = 0}^{N-1} \bby_{k+1/2}$ 
which gives us the following convergence result:
\begin{align}
\left[\max_{\bby: (\hat{\bbx}_N, \bby) \in \mathcal{D}} f(\hat{\bbx}_N, \bby) - f^{\star}\right] + \left[f^{\star} - \min_{\bbx: (\bbx, \hat{\bby}_N) \in \mathcal{D}} f(\bbx, \hat{\bby}_N)\right] \leq \frac{DL \left(16 + \frac{33}{2(1-\sigma^2)} \right)}{N}, \nonumber
\end{align}
where $f^{\star} = f(\bbx^*, \bby^*)$.

\bibliography{references}
\bibliographystyle{icml2019}

\end{document}